\documentclass[a4paper,twoside,10pt]{article}

\newcommand{\mytitle}{The inverse problem for orthotropic media in PS-OCT}
\title{The inverse scattering problem for orthotropic media in Polarization-sensitive Optical Coherence Tomography}

\author{P. Elbau$^1$\\{\footnotesize\href{mailto:peter.elbau@univie.ac.at}{peter.elbau@univie.ac.at}}
\and L. Mindrinos$^1$\\{\footnotesize\href{mailto:leonidas.mindrinos@univie.ac.at}{leonidas.mindrinos@univie.ac.at}}
\and O. Scherzer$^{1,2}$\\{\footnotesize\href{mailto:otmar.scherzer@univie.ac.at}{otmar.scherzer@univie.ac.at}}}
\usepackage[utf8]{inputenc}
\usepackage{amsmath}
\usepackage{longtable}
\usepackage{amssymb}
\usepackage{mathrsfs}
\usepackage{mathtools}
\usepackage{siunitx}
\usepackage[a4paper,centering,bindingoffset=0cm,marginpar=0cm]{geometry}
\usepackage{titlesec}
\usepackage[font=footnotesize,format=plain,labelfont=sc,textfont=sl,width=0.75\textwidth,labelsep=period]{caption}
\titleformat{\section}{\filcenter\sc\large}{\thesection.\;}{0em}{}
\titleformat{\subsection}[runin]{\bf}{\thesubsection.\;}{0em}{}[.]

\newpagestyle{headers}%
{%
	\headrule
	\sethead%
		[\thepage]%
		[\footnotesize\sc P.~Elbau, L.~Mindrinos, and O.~Scherzer]%
		[]%
		{}%
		{\footnotesize\sc{\mytitle}}%
		{\thepage}
}
\usepackage[pdftex,colorlinks=true,linkcolor=blue,citecolor=green,urlcolor=blue,bookmarks=true,bookmarksnumbered=true]{hyperref}
\hypersetup
{
    pdfauthor={Peter Elbau, Leonidas Mindrinos, Otmar Scherzer},
    pdfsubject={Parameter Estimation of Orthotropic Media in Polarization-Sensitive Optical Coherence Tomography},
    pdftitle={\mytitle},
    pdfkeywords={Optical Coherence Tomography, Electromagnetic Scattering, Born Approximation}
}

\usepackage[hyperref,amsmath,thmmarks]{ntheorem}
\usepackage{aliascnt}
\theoremstyle{break}
\theoremseparator{.}
\theoremheaderfont{\kern-1em\normalfont\bfseries}
\newaliascnt{assumption}{lemma}
\newtheorem{assumption}[assumption]{Assumption}
\aliascntresetthe{assumption}

\newaliascnt{theorem}{lemma}
\newtheorem{theorem}[theorem]{Theorem}
\aliascntresetthe{theorem}

\newaliascnt{proposition}{lemma}
\newtheorem{proposition}[proposition]{Proposition}
\aliascntresetthe{proposition}

\newaliascnt{corollary}{lemma}

\aliascntresetthe{corollary}

\newaliascnt{invpro}{lemma}

\aliascntresetthe{invpro}

\theorembodyfont{\normalfont}
\newaliascnt{definition}{lemma}
\newtheorem{definition}[definition]{Definition}
\aliascntresetthe{definition}

\theorembodyfont{\normalfont}
\newaliascnt{example}{lemma}
\newtheorem{example}[example]{Example}
\aliascntresetthe{example}

\theorembodyfont{\normalfont\itshape}
\theoremseparator{:}
\theoremheaderfont{\kern-1em\normalfont\itshape}
\newaliascnt{remark}{lemma}
\newtheorem{remark}[remark]{Remark}
\aliascntresetthe{remark}

\theoremstyle{nonumberplain}
\theorembodyfont{\normalfont}
\theoremsymbol{\ensuremath{\square}}
\newtheorem{proof}{Proof}
    \usepackage[ntheorem]{empheq}

\usepackage{bbm}
\usepackage{bm}
\newcommand{\R}{\mathbbm{R}}

\newcommand{\C}{\mathbbm{C}}

\renewcommand{\b}{\bm}

\newcommand{\gb}[1]{\ensuremath{\mbox{\boldmath $ #1 $}}} % bold type for Gree letters

\newcommand{\tpsi}{\b{\tilde{\psi}}}

\newcommand{\e}{\mathrm e}
\renewcommand{\i}{\mathrm i}
\renewcommand{\d}{\,\mathrm d}
\newcommand{\id}{\mathbbm{1}}

\let\RE\Re
\let\Re=\undefined
\DeclareMathOperator{\Re}{\RE e}
\let\IM\Im
\let\Im=\undefined
\DeclareMathOperator{\Im}{\IM m}

\DeclareMathOperator{\curl}{\mathbf{curl}}
\let\div=\undefined
\DeclareMathOperator{\div}{div}
\DeclareMathOperator{\grad}{\mathbf{grad}}
\DeclareMathOperator{\supp}{supp}

\newcommand{\E}[1]{{\b E}^{#1}}
\newcommand{\FE}[1]{\widehat{\b E}^{#1}}

\newcommand{\FbE}{\widehat{\b E}}
\newcommand{\bE}{{\b E}}

\newcommand{\Einit}{{\b E}^{i}}
\newcommand{\FEinit}{\widehat{\b E}^{i}}

\newcommand{\greenkernel}{G}
\newcommand{\greentensor}{\b G}
\newcommand{\greenoperator}{\bm{\mathcal{G}}}
\newcommand{\greenoperatori}{\bm{\mathcal{G}}^\infty}

\newcommand{\bchi}{\gb \chi} 

\usepackage{xcolor}
\definecolor{darkgreen}{rgb}{0.0, 0.5, 0.0}

\allowdisplaybreaks[1]

\begin{document}

%%%%%%%%%%%%%%%%%%%%%%%%%%%%%%%%%%%%%%%%%%%%%%%%%%
%%% Maketitle
%%%%%%%%%%%%%%%%%%%%%%%%%%%%%%%%%%%%%%%%%%%%%%%%%%
\maketitle
\hspace*{1em}
\parbox[t]{0.49\textwidth}{\footnotesize
\hspace*{-1ex}$^1$Computational Science Center\\
University of Vienna\\
Oskar-Morgenstern-Platz 1\\
A-1090 Vienna, Austria}
\parbox[t]{0.4\textwidth}{\footnotesize
\hspace*{-1ex}$^2$Johann Radon Institute for Computational\\
and Applied Mathematics (RICAM)\\
Altenbergerstra{\ss}e 69\\
A-4040 Linz, Austria}

\vspace*{2em}
%%% End: Maketitle

%%%%%%%%%%%%%%%%%%%%%%%%%%%%%%%%%%%%%%%%%%%%%%%%%%
%%% Abstract
%%%%%%%%%%%%%%%%%%%%%%%%%%%%%%%%%%%%%%%%%%%%%%%%%%
\begin{abstract}
In this paper we provide a mathematical model for imaging an anisotropic, orthotropic medium  with Polarization-Sensitive Optical Coherence Tomography (PS-OCT).  The imaging problem is formulated as an inverse scattering problem in three dimensions  for reconstructing the electrical susceptibility of the medium using Maxwell's equations. Our reconstruction method is based on the second-order Born-approximation of the electric field.
\end{abstract}
%%% End: Abstract

%%%%%%%%%%%%%%%%%%%%%%%%%%%%%%
%%% Introduction
%%%%%%%%%%%%%%%%%%%%%%%%%%%%%%
\section{Introduction}
Optical Coherence Tomography (OCT) is an imaging technique producing high-resolution images of the inner structure of biological tissues. 
Standard OCT uses broadband, continuous wave light for illumination and the images are obtained by measuring the time delay and the 
intensity of the backscattered light from the sample. 
For a detailed description of OCT systems we refer to the books \cite{BouTea02,DreFuj15} and for a mathematical modeling we refer to \cite{ElbMinSch15}.

Apart form standard OCT, there exist also functional OCT techniques such as the Polarization-Sensitive OCT (PS-OCT) which considers the 
differences in the polarization state of light to determine the optical properties of the sample. PS-OCT is based on Polarization-Sensitive Low Coherence Interferometry established by Hee 
\textit{et al.} \cite{HeeHuaSwaFuj92} and then first applied to produce two-dimensional OCT images  \cite{BoeMilGemNel97,BoeSriMalCheNel98}. In this work, we consider the basic scheme of a PS-OCT system which consists of a Michelson  interferometer with the addition of polarizers and quarter-wave plates (QWP).

More precisely, a linear polarizer is added after the source and the linear 
(horizontal or vertical) polarized light is split into two identical parts by a polarization-insensitive beam splitter (BS).
In the reference arm the light is reflected by a mirror and in the sample arm the light is incident on the medium. At the BS, the back-reflected beam and the backscattered light from the sample, in an arbitrary polarization state, are recombined. The recombined light passes 
through a polarizing BS which splits the output signal into its horizontal and vertical components to be measured at two different photo 
detectors. See \autoref{Fig1} for an illustration of this setup.

\begin{figure}
\begin{center}
\includegraphics[scale=0.8]{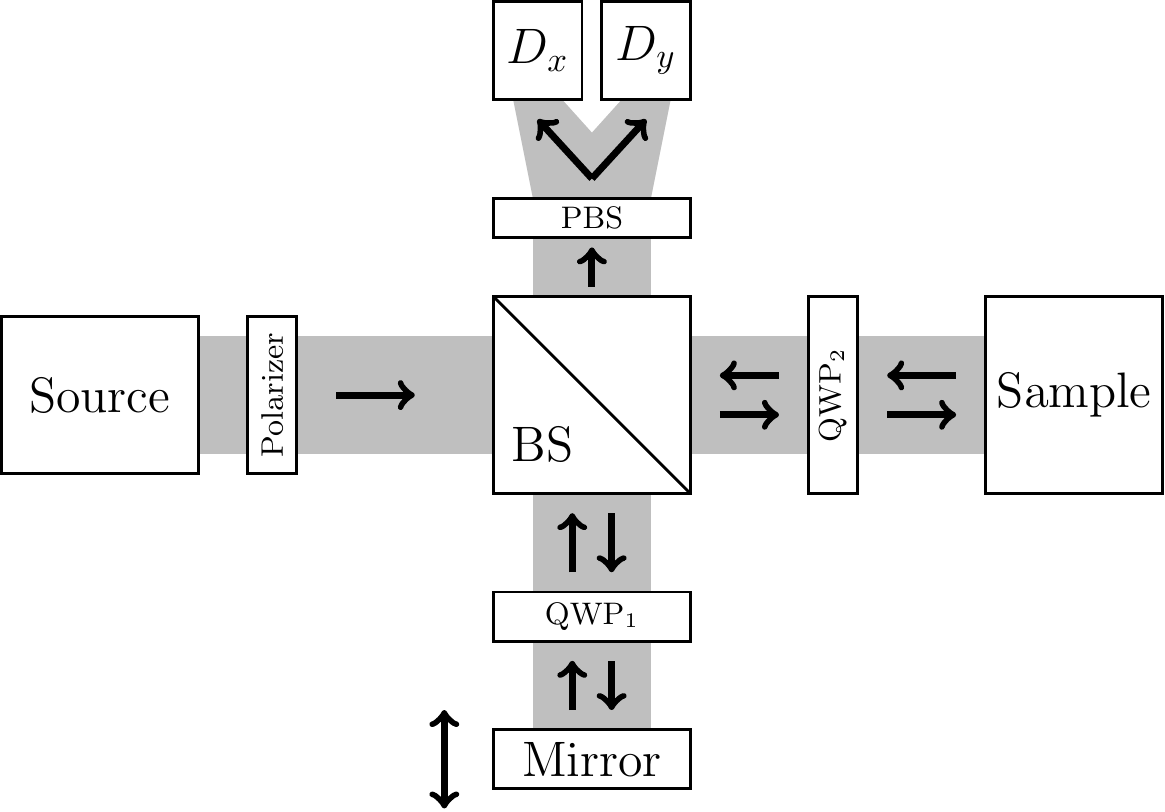}
\caption{Schematic representation of the light travelling in a time-domain PS-OCT system. In the reference and sample arms are placed quarter-wave plates (QWP) at specific orientations.}\label{Fig1}
\end{center}
\end{figure}

To describe the change in the polarization state of the light due to its propagation into the sample we adopt the analysis based on the theory 
of electromagnetic waves  scattered by anisotropic inhomogeneous media \cite{ColKre98, WolFol89}.  We assume that the dielectric medium is 
linear and anisotropic. In addition, we impose the property that the medium is invariant under reflection by the $x_1 - x_2$  plane. A medium with this property is called orthotropic in the  mathematical community \cite{CakCol14} or monoclinic in the material science community \cite{Tor02}.

The medium is also considered as weakly scattering and we present the solution in the accuracy of the second-order Born-approximation. As we are going to see later, we consider higher-order approximation in order to be able to recover all the material parameters. We describe the 
change in the polarization state of the light by the Jones matrix formalism which is applicable since OCT detects the coherent part of the 
electric field of the backscattered light \cite{JiaWan02}. As in standard OCT systems, the backscattered light is detected in the far 
field. 

In the medical community, the sample is usually described by a general retarder and the change in the polarization state of the light returning from the sample can be modelled by a Jones matrix \cite{HitGotStiPirFer01, JiaWan02}. However, even though the produced images are satisfactory they are mainly used qualitatively. The usage of these images comes only secondarily to quantify the optical parameters by image processing techniques.

In this work we are interested in the quantitative description of PS-OCT. To do so, we have first to describe mathematically the system properly. Thus, we represent  the polarized scattered field as solution to the full-wave Maxwell's equations. This has not yet been applied to PS-OCT, since for isotropic media, the Born-approximation decouples the effects of the optical properties of the sample from the polarization state of the scattered field. 
However, this analysis for anisotropic media provides enough information to reconstruct the electric susceptibility of the medium. The scattered field satisfies then an integral equation of Lippmann-Schwinger type. Under the far-field approximation and the assumption of a homogeneous background medium we obtain a system of integral equations for the unknown optical parameters.

In the mathematical literature, the scattering problem by anisotropic objects has been widely considered over the last decades 
\cite{BekUma89,GenWuLi03,GraUsl84,PapUzuCap90}. 
Recently, the connection between the inverse problem to reconstruct the refractive index and the interior transmission problem has been 
investigated \cite{CakColMonSun10, CakHad07}. For the specific case of an orthotropic medium we refer the reader to the book 
\cite{CakCol14} and to \cite{ColKreMon97, Pot99} for results concerning the uniqueness and existence of solutions of the inverse problem.

The paper is organized as follows: In \autoref{direct}, we derive the integral representation of the scattered field for an orthotropic medium in the accuracy of the second-order Born-approximation in the far-field zone. In \autoref{psoct}, we describe mathematically the standard PS-OCT system using the Jones matrix formalism and we derive the expression for the cross-spectral density. The system of equations for all the components of the susceptibility is presented in the last section using two incident illuminations.

\subsection*{Notation}
In this paper, we use the following conventions:
\begin{itemize}
 \item Let $f : \R \to \C$ be integrable, then the one-dimensional Fourier-transform is defined by   
       \[ \hat f(\omega) = \int_{\R} f(t)\e^{\i\omega t}\d t, \quad \text{ for all } \omega \in \R\;.\] 
 \item Let $f : \R \to \C$ be integrable, then the one-dimensional inverse Fourier-transform is defined by   
       \[ \check f(t) = \frac1{2\pi}\int_{\R} f(\omega)\e^{-\i \omega t}\d \omega, \quad  \text{ for all } t \in \R\;.\]
       \item Let $f : \R^3 \to \C$ be integrable, then the three-dimensional Fourier-transform is defined by   
       \[ \tilde f(\b k) = \int_{\R^3} f(\b x)\e^{-\i  \left<\b k,\b x\right>} \d \b x , \quad \text{ for all } \b x \in \R^3\;.\] 
\end{itemize}
%%% End

%%%%%%%%%%%%%%%%%%%%%%%%%%%%%%
%%% The direct scattering problem
%%%%%%%%%%%%%%%%%%%%%%%%%%%%%%
\section{The direct scattering problem}\label{direct}
In absence of external charges and currents, the propagation of electromagnetic waves 
in a non-magnetic medium is mathematically described by Maxwell's equations relating the electric and magnetic 
fields $\bE: \R \times \R^3 \to \R^3$ and $\b H: \R \times \R^3 \to \R^3$ 
and the electric displacement $\b D: \R \times \R^3 \to \R^3$ by
\begin{subequations}\label{eqMaxwell}
\begin{align}
\div \b D(t,\b x) &= 0,\quad &&t\in\R,\;\b x\in\R^3, \label{eqMaxwellGauss}\\
\div \b H(t,\b x) &= 0,\quad &&t\in\R,\;\b x\in\R^3, \\
\curl \bE(t,\b x) &= -\frac1c\frac{\partial \b H}{\partial t}(t,\b x),\quad &&t\in\R,\;\b x\in\R^3, \label{eqMaxwellFaraday}\\
\curl \b H(t,\b x) &= \frac1c\frac{\partial \b D}{\partial t}(t,\b x),\quad &&t\in\R,\;\b x\in\R^3, \label{eqMaxwellAmpere}
\end{align}
\end{subequations}
where $c$ is the speed of light.  Maxwell's equations are not sufficient to uniquely determine the fields $\b D, \bE$ and $\b H$. Therefore additional material 
parameters have to be specified:
\begin{definition}\label{deDielectricum}
\begin{itemize}
 \item An anisotropic medium is called \emph{linear dielectric} if there exists a function, called the \emph{electric susceptibility},
       \begin{equation*}
          \bchi \in C^\infty_{\mathrm c}(\R\times\R^3 ; \R^{3\times 3}), \text{    with   } 
          \bchi(\tau,\b x)=0 \text{ for all } \tau<0, \b x\in\R^3,
       \end{equation*}
       satisfying
       \begin{equation}\label{eqDielectricum}
          \b D (t,\b x) = \bE (t, \b x) + \int_{\R} \bchi(\tau,\b x)\bE(t-\tau,\b x)\d\tau\;.
       \end{equation}
 \item A linear dielectric medium is called \emph{orthotropic} \cite{CakCol14,ColKreMon97} if 
      % \begin{itemize}
      % \item $\bchi(t,\b x)$ is symmetric for all $t \in \R, \,\,\b x \in \R^3$ and
       %\item 
	it admits the special symmetric form 
             \begin{equation}\label{eqortho}
                 \bchi = \begin{pmatrix}
                           \chi_{11} & \chi_{12} & 0 \\
                           \chi_{12} & \chi_{22} & 0 \\
                            0 & 0  & \chi_{33}
                          \end{pmatrix} .
             \end{equation}
%\end{itemize}
\end{itemize}
\end{definition}

Application of the Fourier transform to Maxwell's equations \eqref{eqMaxwell} and taking into account 
\eqref{eqDielectricum}, it follows that the Fourier-transform $\FbE$ of $\bE$ satisfies the 
\emph{vector Helmholtz equation}
\begin{equation}\label{eqVectorHelmholtz}
\curl \curl \FbE (\omega,\b x) - \frac{\omega^2}{c^2}(\id+\hat{\bchi}(\omega,\b x)) \FbE(\omega, \b x) = 0,
\quad\omega\in\R,\;\b x\in\R^3.
\end{equation}

\begin{definition}\label{de:cif}
We call an electric field $\Einit:\R\times\R^3\to\R^3$ a \emph{causal initial field} (CIF) with respect to some domain 
$\Omega \subseteq \R^3$ if
\begin{enumerate}
 \item its Fourier transform with respect to time solves Maxwell's equations \eqref{eqMaxwell} with a susceptibility $\bchi = 0 $, that is,
       \begin{equation}\label{eqVectorHelmholtzIncident}
          \curl \curl \FEinit (\omega,\b x)- \frac{\omega^2}{c^2} \FEinit (\omega,\b x)= 0, \text{   and   } 
          \div \FEinit (\omega,\b x) = 0,\quad \omega\in\R,\;\b x\in\R^3,
       \end{equation}
 \item and satisfies $\supp \Einit (t,\cdot) \cap \Omega = \emptyset \text{ for every } t \le 0$. 
\end{enumerate}
\end{definition}
The second condition means that $\Einit$ does not interact with the medium contained in $\Omega$ until the time $t=0.$
\begin{example} \label{ex:E0}
Let $\Omega \subset \R^3$ be an open set, such that $\supp \bchi (t,\cdot) \subset \Omega$ for all $t \in\R.$  
Moreover, let $\b q\in\R^2\times\{0\}$ (denoting the polarization vector), $f \in C^\infty_{\mathrm c}(\R)$ and 
\begin{equation}\label{eq:ill}
 \E{0} (t,\b x)= \b q f(t+\tfrac{x_3}c),
\end{equation}
such that 
\[
\supp \E{0} (t,\cdot) \cap \Omega = \emptyset \text{ for every } t \le 0.
\]
Then $\E{0}$ is a CIF.
\end{example}
\begin{proof}
To see this note that for arbitrary $\b q \in \R^3$ we get
\begin{align*}
 \curl \curl \E{0} &= \curl \left( \frac1c f'(t+\tfrac{x_3}c) \,\b e_3 \times \b q\right) = \frac1{c^2} f''(t+\tfrac{x_3}c) \,\b e_3 \times (\b e_3 \times \b q ) \\
 &= -\frac1{c^2} f''(t+\tfrac{x_3}c) \b q =-\frac1{c^2} \partial_{tt} \E{0} .
\end{align*}
And for the particular choice $\b q\in\R^2\times\{0\}$ we even have that $\div \E{0}=0$. This shows that $\E{0}$ is a solution 
of Maxwell's equation. The second assertion is an immediate consequence of the second assumption.
\end{proof}

\begin{theorem}
Let $\Einit$ be a CIF-function and assume that the susceptibility $\bchi$ represents a dielectric, orthotropic medium. Then, 
\begin{enumerate}
\item there exists a solution $\bE$ (together with $\b H$) of Maxwell's equations~\eqref{eqMaxwell} which satisfies
      \begin{equation}\label{eqInitialCondition}
        \bE(t,\b x) = \Einit(t,\b x),\quad t \le 0,\;\b x\in\R^3.
      \end{equation}
\item For every $\b x \in \R^3$ the function 
      \begin{equation*}
      \begin{aligned}
         g: \R & \to \C\,,\\
        \omega &\mapsto (\FbE - \FEinit)(\omega,\b x),
      \end{aligned}
      \end{equation*}
      can be extended to a square integrable, holomorphic function on the upper half plane 
      $$\mathbbm{H}=\{ \omega\in \C \mid \Im{(\omega)} >0 \}.$$
\item $\FbE$ solves the \emph{Lippmann--Schwinger integral equation}
      \begin{equation}\label{eqLippmannSchwinger}
      \begin{aligned}
      \FbE(\omega,\b x) 
      &= \FEinit(\omega,\b x)+\left(\frac{\omega^2}{c^2} \id +\grad\div\right)\int_{\R^3}\greenkernel (\omega, \b x-\b y)\hat{\bchi}(\omega,\b y)\FbE(\omega,\b y)\d \b y\\
      &=: \FEinit(\omega,\b x) + \greenoperator [\hat{\bchi} \FbE](\omega,\b x)\,,
      \end{aligned}
      \end{equation}
      where       
      \[\greenkernel (\omega, \b x) = \frac{\e^{\i\frac\omega c| \b x|}}{4\pi|\b x|}, \quad \b x \neq 0, \, \omega \in \R\]
      is the fundamental solution of the scalar Helmholtz equation.
\end{enumerate}
\end{theorem}
The integral operator $\greenoperator$ is strongly singular and we address its properties in the last section.

\begin{proof}\mbox{}\\\vspace*{-2em} 

\begin{enumerate}
 \item[\em ii.] From the initial condition~\eqref{eqInitialCondition} it follows for every solution $\bE$ of Maxwell's equations~\eqref{eqMaxwell} which fulfills \eqref{eqInitialCondition} that the inverse Fourier-transform of $g$ satisfies 
\begin{equation*}
 \check{g}(t)=0 \text{ for all } t \le 0.
\end{equation*}
Thus, the result is a direct consequence of the Paley--Wiener theorem \cite{Pap62}.
 \item[\em i.] The electric field $\FbE$ is uniquely defined by \eqref{eqVectorHelmholtz} together with the assumption that the function $g$ can be for every $x\in\R^3$ extended to a square integrable, holomorphic function on the upper half plane.
\item[\em iii.] The solution of equation \eqref{eqVectorHelmholtz} can be written as the solution of the integral equation \eqref{eqLippmannSchwinger}, see \cite{CakCol14,Pot00}.
\end{enumerate}

\end{proof}

\subsection{Born and Far-field approximation}
We assume that the medium is weakly scattering, meaning that $\hat{\bchi}$ is sufficiently small \cite{Che90,ColKre98}
such that the incident field $\FEinit$ is significantly larger than $\FbE-\FEinit$. 

\begin{definition}
The \emph{first order} Born-approximation of the solution $\FbE$ of the Lippmann-Schwinger equation \eqref{eqLippmannSchwinger} 
is defined by 
\begin{equation}\label{eqBorn1}
\FE{1} = \FEinit + \greenoperator [\hat{\bchi} \FEinit ].
\end{equation}
The \emph{second order} Born-approximation is defined by 
\begin{equation}\label{eqBorn2}
\FE{2} = \FEinit + \greenoperator [\hat{\bchi} \FE{1} ].
\end{equation}
\end{definition}
Inserting \eqref{eqBorn1} into \eqref{eqBorn2} gives
\begin{equation}\label{eqBorn2com}
\FE{2} = \FEinit + \greenoperator [\hat{\bchi} \FEinit] + \greenoperator \left[\hat{\bchi}\greenoperator [\hat{\bchi} \FEinit] \right],
\end{equation}
or in coordinate writing
\begin{equation}\label{eqBorn2b}
\begin{aligned}
\FE{2}(\omega,\b x) &= \FEinit(\omega,\b x)+ 
\frac{\omega^2}{c^2}\int_{\R^3} \greentensor (\omega, \b x-\b y)\hat{\bchi}(\omega,\b y)\FEinit(\omega,\b y)\d \b y \\
&\phantom{=}+ \frac{\omega^4}{c^4}\int_{\R^3} \int_{\R^3} \greentensor (\omega, \b x-\b y)\hat{\bchi}(\omega,\b y) 
\greentensor (\omega, \b y-\b z)\hat{\bchi}(\omega,\b z) \FEinit(\omega,\b z)\d \b z\d \b y ,
\end{aligned}
\end{equation}
where now $\b G$ is the Green tensor of Maxwell's equations \cite{Had04, HazLen96}
\[
\b G (\omega , \b x - \b y) = G (\omega , \b x - \b y) \id + \frac{c^2}{\omega^2} \grad\div (G (\omega , \b x - \b y) \id).
\]

The physical meaning of the second order Born-approximation is that at a point $\b x$ the total field  $\FE{2}$ contains all single and double scattering events. 

In an OCT setup, the measurements are performed in a distance much bigger compared to the size of the sample. Thus, setting
$\b x = \rho \bm\vartheta, \rho>0$ and $\bm \vartheta \in \mathbb{S}^2,$  we can replace the above expression by its asymptotic 
behaviour for $\rho \to\infty,$ uniformly in $\bm\vartheta,$  see for instance \cite[Equation (4.1)]{ElbMinSch15}, resulting to 
\begin{equation}\label{eqFar}
\FE{2} (\omega,\rho\bm\vartheta) = \FEinit(\omega,\rho\bm\vartheta)+ \greenoperatori [\hat{\bchi} \FEinit ] (\omega,\rho\bm\vartheta) + 
\greenoperatori \left[\hat{\bchi}\greenoperator [\hat{\bchi} \FEinit ] \right] (\omega, \rho\bm\vartheta) .
\end{equation}
Here we have introduced the operator
\begin{equation}\label{operatorinf}
\greenoperatori [\b f ] (\omega , \rho \gb\vartheta ) := -\frac{\omega^2 \e^{\i\tfrac{\omega}c \rho}}{4\pi \rho c^2} \int_{\R^3} \bm\vartheta \times \left( \bm\vartheta \times \b f (\omega,\b y) \right) \e^{-\i\frac\omega c\left<\bm\vartheta ,\b y\right>} \d \b y,
\end{equation}
defined for functions $\b f : \R \times \R^3 \to \R^3$.

%%% End

%%%%%%%%%%%%%%%%%%%%%%%%%%%%%%
%%% Polarized-sensitive OCT
%%%%%%%%%%%%%%%%%%%%%%%%%%%%%%
\section{Polarized-sensitive OCT}\label{psoct}
We describe the standard PS-OCT system in the context of a Michelson interferometer first presented by Hee \textit{et al.} \cite{HeeHuaSwaFuj92}. 

The detector array is given by $\mathcal D=\R^2\times\{d\}$ with $d>0$ sufficiently large. Moreover, we specify the CIF function to be $\E{0}$ as defined in \autoref{ex:E0} and we assume that $\E{0} (t,\b x)=0$ for $t\ge0$  and $\b x\in\mathcal D.$ 

We describe now the change in the polarization state of the light through the PS-OCT system. The effect of the polarization-insensitive beam splitter (BS) is not considered in this work since it only reduces the intensity of the beam by a constant factor. For simplicity, we place the sample and the mirror around the origin and the detector at the BS, for more details see \cite[Section 3.3]{ElbMinSch15}. The BS splits the light into two identical beams entering both arms of the interferometer.

\begin{figure}
\begin{center}
\includegraphics[scale=0.9]{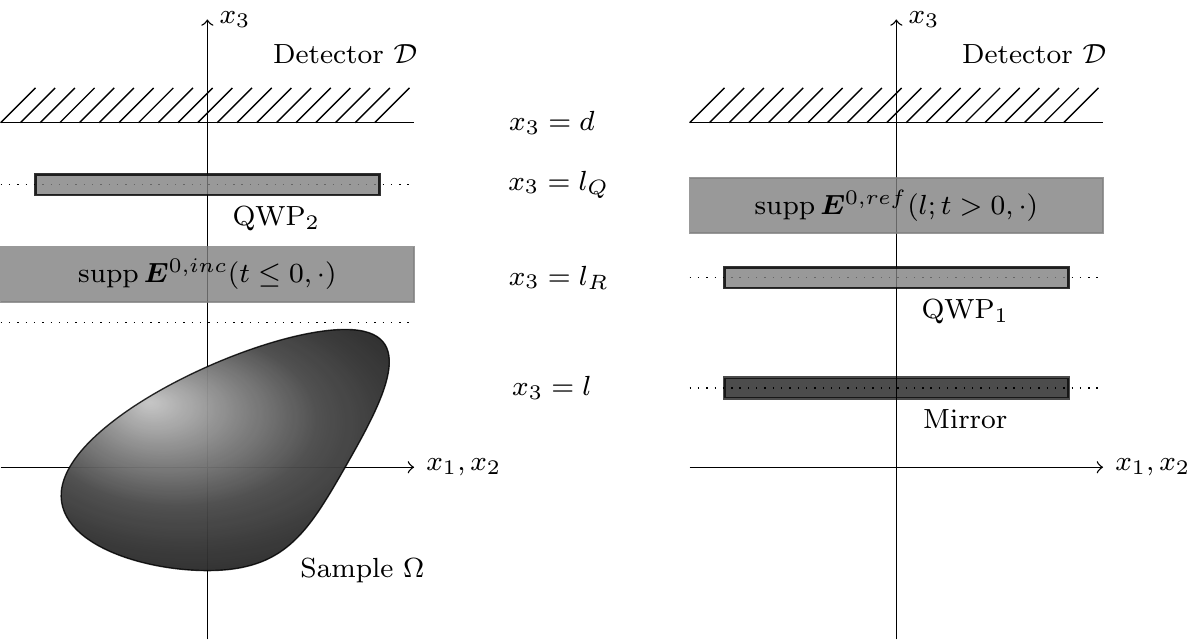}
\caption{The two scattering problems in PS-OCT. On the left picture the incoming light in the sample arm passes through a QWP and is incident on the medium. On the right picture, in the reference arm, the light is back-reflected by a mirror (passing twice a QWP).}\label{Fig2}
\end{center}
\end{figure}

\begin{description}
\item[Reference arm:] The light (at some negative time) passes through a zero-order quarter-wave plate (QWP) oriented at angle $\phi_1$ to the incident linear polarization. It is reflected by a perfect mirror placed in $x_3=l$ and then passes through the QWP again, at 
time $t = 0,$ see the right picture in \autoref{Fig2}. We formulate this process as a linear operator 
\begin{equation}\label{operator_ref}
\bm{\mathcal{J}}_l [\E{0} ] (t, \b x) = \E{0,ref}  (l;t,\b x),
\end{equation}
to be specified later. Then, the reference field 
$\bE^l$ takes the form
\begin{equation}\label{eqReflectedField}
 \bE^l(t,\b x) = \begin{cases}
    \E{0} (t,\b x) + \E{0,ref}  (l;t,\b x) , & \text{if } t> 0,\;\b x_3 > l_R ,\\
  \E{0}(t,\b x), & \text{if } t\le 0,\;\b x_3 > l_R .
  \end{cases}
\end{equation}

\item[Sample arm:] The incoming light passes through a QWP (oriented at a different angle $\phi_2 )$ at some time $t<0,$ placed in the plane given by the equation $x_3 = l_Q.$ This process results to a field 
\begin{equation}\label{operator_sam}
\bm{\mathcal{J}} [\E{0} ] (t, \b x) = \E{0,inc}  (t,\b x) ,
\end{equation}
that until $t = 0$ does not interact with the medium, see the left picture in \autoref{Fig2}.

\item[Detector:] The electric field $\bE$ which is obtained by illuminating the sample with the incident field 
$\E{0,inc}$ is combined with the reference field $\bE^l .$ We assume here that the backscattered light 
does not pass through the QWP again. At every point on the detector surface $\mathcal D$ we measure the two intensities \cite{ElbMinSch15}
\begin{equation*}%\label{eqIntensity}
I_j(l,\b\xi)=\int_{0}^\infty E_j (t,\b \xi) E^l_j (t,\b \xi)\d t,\quad\b\xi\in\mathcal D,\;j\in\{1,2\}.
\end{equation*}
\end{description}

We assume that we do not measure the incident fields at the detector, meaning  $\E{0}(t,\b\xi)=\E{0,inc}(t,\b\xi)=0$ for $t> 0$ and $\b\xi\in\mathcal D$ and recalling \eqref{eqReflectedField} we obtain $
\b E^l-\b E^0 = 0$ for $t\leq 0,$ resulting to
\begin{align}\label{intensity}
 I_j(l,\bm\xi) &= \int_{0}^\infty(E_j- E^{0,inc}_j)(t,\bm\xi)(E^l_j -E^{0}_j)(t,\bm\xi)\d t 
\nonumber \\
&= \int_{\R}(E_j- E^{0,inc}_j)(t,\bm\xi)(E^l_j -E_j^{0})(t,\bm\xi)\d t .
\end{align}

We use Plancherel's theorem, and since $\bE \in \R^3$ it follows that $\FbE(-\omega ,\cdot) = \overline\FbE (\omega ,\cdot).$ Thus, the above formula can be rewritten as
\begin{equation}\label{data_freq}
\begin{aligned}
I_j(l,\bm\xi) &= 
\frac1{2\pi}\int_{\R}(\hat E_j-\hat E_j^{0,inc})(\omega,\bm\xi)(\overline{\hat E^l_j -\hat E_j^{0}})(\omega,\bm\xi)\d \omega \\
&= 
\frac1{2\pi}\int_{-\infty}^0 (\overline{\hat E_j-\hat E_j^{0,inc}})(-\omega,\bm\xi)(\hat E^l_j -\hat E_j^{0})(-\omega,\bm\xi)\d \omega \\
&\phantom{=}+\frac1{2\pi}\int_0^\infty (\hat E_j-\hat E_j^{0,inc})(\omega,\bm\xi)(\overline{\hat E^l_j -\hat E_j^{0}})(\omega,\bm\xi)\d \omega \\
&= \frac1{\pi} \Re \int_0^\infty (\hat E_j-\hat E_j^{0,inc})(\omega,\bm\xi)(\overline{\hat E^l_j -\hat E_j^{0}})(\omega,\bm\xi)\d \omega .
\end{aligned}
\end{equation}

\subsection{Jones Calculus} Here we describe the operators $\bm{\mathcal{J}}_l$ and $\bm{\mathcal{J}}$, introduced in \eqref{operator_ref} and \eqref{operator_sam}, respectively. We consider the fields in the frequency domain. Then, for positive frequencies we can apply the Jones matrix method (keeping also the zero third component of the fields) in order to model the effect of the QWP's on the polarization state of light. We assume that the properties of the QWP's are frequency independent and that the light is totally transmitted through the plate surface.

\begin{definition}
We define
\begin{equation}
\begin{aligned}
\bm{\mathcal{J}}_l [\b v ] (\omega, \b x) &=  \b J^2 (\phi_1 ) \b v  (\omega,\b x) \, \e^{\i\tfrac{\omega}c 2(x_3 - l)} ,  & \mbox{for } \omega > 0,\\
\bm{\mathcal{J}} [\b v ] (\omega, \b x) &= \b J (\phi_2) \b v  (\omega,\b x), & \mbox{for } \omega > 0,
\end{aligned}
\end{equation}
where 
\begin{equation*} %\label{eqMatrix}
 \b J (\phi)= \begin{pmatrix} 
\cos \phi & -\sin \phi & 0\\ 
\sin \phi & \phantom{-}\cos \phi & 0  \\
0 & 0 & 1
\end{pmatrix}  \begin{pmatrix} 
1 & \phantom{-}0 & 0 \\ 
0 & -\i & 0 \\
0 & \phantom{-}0 & 1
\end{pmatrix}  \begin{pmatrix} 
\phantom{-}\cos \phi & \sin \phi & 0\\ 
-\sin \phi & \cos \phi & 0  \\
0 & 0 & 1
\end{pmatrix},
\end{equation*}
is the rotated Jones matrix for a QWP with the fast axis oriented at angle $\phi$ \cite{GerBur75}.
\end{definition}

The above definition summarizes what we described before: In the reference arm, the incoming field passes through the QWP (at angle $\phi_1)$ is reflected by the mirror and then passes through the QWP again. The field travels additionally the distance $2(x_3 - l).$ In the sample arm, the field passes only through the QWP at angle $\phi_2 .$

We consider the PS-OCT system, presented first in \cite{HeeHuaSwaFuj92} and then considered in \cite{HitGotStiPirFer01, SchoColMaiSil98}, where $\phi_1 = \pi/8$ and $\phi_2 = \pi/4.$ Then, we obtain
\begin{equation}\label{reference}
\begin{aligned}
\FE{0,ref}  (l;\omega,\b x) &= \bm{\mathcal{J}}_l [\FE{0} ] (\omega, \b x) = \bm\eta \hat f(\omega) \,\e^{\i\tfrac{\omega}c (x_3 - 2l)},  & \mbox{for } \omega > 0,\\
\FE{0,inc}  (\omega,\b x) &= \bm{\mathcal{J}} [\FE{0} ] (\omega, \b x) =  \b p \hat f(\omega) \, \e^{-\i\tfrac{\omega}c x_3 } ,  & \mbox{for } \omega > 0,\\
\end{aligned}
\end{equation}
where $\bm \eta =  \b J^2 (\pi/8) \, \b q$ and $\b p = \b J (\pi/4) \, \b q.$ We observe that $\FE{0,ref} $ is still linearly polarized at angle $\pi/4$ with the linear (horizontal or vertical) initial polarization state and $\FE{0,inc}$ describes a circularly polarized light.

Now we can define our approximated data. We approximate in \eqref{data_freq} the term $\hat E_j-\hat E_j^{0,inc}$ by $\hat E^{2}_j-\hat E_j^{0,inc}$ and for the term $\hat E^l_j -\hat E_j^{0}$ we consider \eqref{eqReflectedField} and \eqref{reference}.

\begin{definition}
We call
 \begin{equation}\label{eqIntensity}
I_j^{2}(l,\bm\xi) = 
\frac{\eta_j}{\pi} \Re \int_0^\infty (\hat E^{2}_j-\hat E_j^{0,inc})(\omega,\bm\xi)
\hat f(-\omega) \e^{\i\tfrac{\omega}c (2l-\xi_3)}\d \omega .
\end{equation}
the second order approximated measurement data of OCT.
\end{definition}

%%% End

%%%%%%%%%%%%%%%%%%%%%%%%%%%%%%
%%% The inverse scattering problem
%%%%%%%%%%%%%%%%%%%%%%%%%%%%%%
\section{The inverse problem of recovering the susceptibility}
The problem we address here, is to recover $\hat{\bchi}$ from the knowledge of $\b I^2 (l,\bm\xi)$ for $l\in\R,$ $\bm\xi\in \mathcal D.$ First, we show that the measurements provide us with expressions which depend on $\hat{\bchi}$ non-linearly.

\begin{proposition}\label{propo1}
Let $\E{0}(t,\b x)$ be given by the form~\eqref{eq:ill} with $q_3=0$ and let the measurement data $I^2_j$ be given by \eqref{eqIntensity}. Then, for every $\omega\in\R_+\setminus\{0\}$ with $\hat f(\omega)\ne0$, the expression
\begin{equation}\label{eqProp}
\eta_j  \left[\greenoperatori \left[\hat{\bchi}\left( \b p \, \e^{-\i\frac\omega c  y_3} + \greenoperator [\hat{\bchi} \b p \, \e^{-\i\frac\omega c  z_3} ] \right)\right] \right]_j (\omega , \rho \gb\vartheta ) = \frac1{  c | \hat f(\omega)|^2 }\int_{\R} I^2_j( l,\rho \gb\vartheta)\e^{-\i\frac\omega c(2l-\rho \vartheta_3)}\d l 
\end{equation}
holds for all $j\in\{1,2\}$, and $\gb\vartheta\in \mathbb{S}^2_+ :=\{\gb\mu\in \mathbb{S}^2\mid \mu_3>0\}.$
\end{proposition} 

\begin{proof}
We consider equation \eqref{eqFar} where now $\FE{i}$ is replaced by $\FE{0,inc}$ for $\omega >0$.  Then, we get

\begin{equation}\label{eqBornOper}
(\FE{2}-\FE{0,inc})(\omega,\rho \gb\vartheta) =  \hat f(\omega)\,   \greenoperatori \left[\hat{\bchi}\left( \b p \, \e^{-\i\frac\omega c  y_3} + \greenoperator [\hat{\bchi} \b p \, \e^{-\i\frac\omega c  z_3} ] \right)\right] (\omega, \rho\bm\vartheta) .
\end{equation}
We apply the inverse Fourier transform with respect to $l$ in \eqref{eqIntensity}, to obtain
\begin{equation}
\begin{aligned}
\int_\R I_j^{2}(l,\bm\xi) \e^{-\i \tfrac{\tilde\omega }{c}2l} \d l &= 
\frac{c\eta_j}{2}  \int_0^\infty (\hat E^{2}_j-\hat E_j^{0,inc})(\omega,\bm\xi)
\hat f(-\omega) \e^{-\i\tfrac{\omega}c \xi_3 } \delta (\omega -  \tilde{\omega})\d \omega \\
&\phantom{=}+ \frac{c\eta_j}{2}  \int_0^\infty \overline{(\hat E^{2}_j-\hat E_j^{0,inc})(\omega,\bm\xi)
\hat f(-\omega) \e^{-\i\tfrac{\omega}c \xi_3 }} \delta (\omega +  \tilde{\omega})\d \omega
\end{aligned}
\end{equation}
which for $\tilde{\omega}>0, \,\hat f \ne0$  and $\eta_j \ne 0,$ using that $\b E$ and $f$ are real valued, results to
\[
 (\hat E^{2}_j-\hat E_j^{0,inc})(\omega,\bm\xi) = \frac1{ \eta_j c \hat f(-\omega) }\int_{\R} I^2_j( l,\bm\xi )\e^{-\i\frac\omega c(2l-\xi_3)}\d l.
\]
This identity together with \eqref{eqBornOper}, results asymptotically to \eqref{eqProp}.
\end{proof}

We observe here that we want to reconstruct four four-dimensional functions  from two three-dimensional  measurement data. Thus, we have to consider some additional assumptions on the medium in order to cancel out the lack of dimensions and handle the non-linearity of \eqref{eqProp} with respect to $\hat{\bchi}.$

\begin{assumption}\label{as1}
Specific illumination:  The support of the initial pulse is small enough such that the optical parameters in this spectrum can be assumed constant with respect to frequency. 

Medium: The susceptibility can be decomposed into two parts, a background susceptibility which is constant and assumed to be  known and a part that counts for the local variations of the susceptibility and can be seen as deviation from the constant value. 

Then, the expression \eqref{eqortho} admits the special form
\[
\hat{\bchi} (\omega , \b x)= \bchi_0 + \epsilon \,\gb\psi (\b x),
\]
where
\[
\bchi_0 = \chi_0 \begin{pmatrix}
    1 & 1 & 0 \\
    1 & 1 & 0 \\
    0 & 0  & 1
\end{pmatrix}, \quad \mbox{and} \quad \gb\psi = \begin{pmatrix}
    \psi_{11} & \psi_{12} & 0 \\
    \psi_{12} & \psi_{22} & 0 \\
    0 & 0  & \psi_{33}
\end{pmatrix},
\]
for some known $\chi_0 \in\R$, a small parameter $\epsilon >0$ and $\psi_{ij} \in C^\infty_{\mathrm c}(\R^3 ; \C).$
\end{assumption}

In the following, we consider this type of media, which are typical for biological tissues, and we assume in addition that the behavior of the homogeneous medium ($\epsilon = 0$) is known. Then, as a consequence, also the measured data from PS-OCT are known, let us call them $I_0$, and we can assume the following form for the measurements
\begin{equation}\label{measurements2}
I^2_j(l,\bm\xi) = I_0 + \epsilon M_j(l,\bm\xi), \quad l \in \R, \, \bm\xi \in \mathcal{D}, \, j\in \{1,2\} .
\end{equation}
for some known functions $M_j.$ 

\begin{proposition}
Let the assumptions of \autoref{propo1} and the additional \autoref{as1}
hold. We define $\b v= \tfrac\omega c (\bm\vartheta +\b e_3),\, \gb\vartheta\in \mathbb{S}^2_+ .$ Then, the spatial Fourier transform of the matrix-valued function $\bm\psi : \R^3 \rightarrow \C^{3\times 3},$ satisfies the equations
\begin{equation}\label{eqFinal}
\eta_j \left[ \bm\vartheta \times \left( \bm\vartheta \times \left( \left( \tpsi (\b v) + \bchi_0 \bm{\mathcal{K}} [\tpsi ] (\b v) + \bm{\mathcal{K}}^\dagger [\tpsi ] (\b v) \bchi_0 \right) \b p \right)\right) \right]_j = \tilde m_j (\b v ) , \quad j\in\{1,2\} ,
\end{equation}
where 
\begin{equation}\label{rightside}
\tilde m_j (\b v ) := m_j (\omega , \bm\vartheta ) = -\frac{4\pi\rho c}{  \omega^2| \hat f(\omega)|^2 }\int_{\R} M_j( l,\rho \gb\vartheta)\e^{-\i\frac\omega c(2l-\rho (\vartheta_3-1))}\d l .
\end{equation}
The operators $\bm{\mathcal{K}}$ and $\bm{\mathcal{K}}^\dagger$ are defined by
\begin{equation}\label{Operators}
\bm{\mathcal{K}} [\b f ] (\b v) : =  \int_{\R^3} \b K^{\b z} (\b v; \b k)  \b f (\b k) \d \b k , \quad \bm{\mathcal{K}}^\dagger [\b f ] (\b v) : =  \int_{\R^3} \b f (\b k) \b K^{\b y} (\b v; \b k)   \d \b k .
\end{equation}
for functions $\b f : \R^3 \rightarrow \C^{3\times 3},$ with kernels
\begin{align*}
\b K^{\bm\alpha} (\tfrac\omega c (\bm\vartheta +\b e_3); \b k) = \frac{\omega^2 }{c^2(2\pi)^3} \int_{\Omega}\int_{\Omega}  \greentensor (\omega, \b y-\b z)\e^{-\i\frac\omega c (z_3 + \left<\bm\vartheta,\b y\right>)} \e^{\i\left<\b k,\bm\alpha\right>} \d \b z\d \b y, \quad \bm\alpha = \b z, \b y.
\end{align*}
\end{proposition}

\begin{proof}
We substitute $\hat{\bchi},$ considering \autoref{as1}, and \eqref{measurements2} in  \eqref{eqProp}  and we equate the first order terms $\bm\psi$ and $\b M$ to obtain
\begin{multline}\label{eqPert1}
\eta_j  \left[\greenoperatori \left[\bm\psi \left( \b p \, \e^{-\i\frac\omega c  y_3} + \greenoperator \left[ \bchi_0 \b p  \e^{-\i\frac\omega c z_3}\right]\right) \right] \right]_j (\omega , \rho \gb\vartheta ) 
+ \eta_j  \left[\greenoperatori \left[\bchi_0  \greenoperator \left[ \bm\psi \b p \, \e^{-\i\frac\omega c z_3 }\right] \right] \right]_j (\omega , \rho \gb\vartheta ) \\
=  \frac1{  c | \hat f(\omega)|^2 }\int_{\R} M_j( l,\rho \gb\vartheta)\e^{-\i\frac\omega c(2l-\rho \vartheta_3)}\d l  .
\end{multline}

In order to analyse the left hand side of the above equation we consider the definition \eqref{operatorinf} and the analytic form \eqref{eqBorn2b}. Then, we  rewrite \eqref{eqPert1} as
 \begin{multline*}%\label{eqPert2}
\eta_j \left[ \int_{\Omega} \bm\vartheta \times \left( \bm\vartheta \times \left( \gb\psi (\b y)\b p\right) \right) \e^{-\i\frac\omega c\left<\bm\vartheta +\b e_3,\b y\right>} \d \b y \right.\\
+
\frac{\omega^2 }{c^2} \int_{\Omega} \int_{\Omega} \bm\vartheta \times \left( \bm\vartheta \times \left( \bchi_0 \greentensor (\omega, \b y-\b z)\gb\psi (\b z) \b p \right) \right)\e^{-\i\frac\omega c (z_3 + \left<\bm\vartheta,\b y\right>)} \d \b z\d \b y  \\
+ \left.\frac{\omega^2 }{c^2} \int_{\Omega} \int_{\Omega} \bm\vartheta \times \left( \bm\vartheta \times \left( \gb\psi (\b y) \greentensor (\omega, \b y-\b z)\bchi_0 \b p\right) \right)\e^{-\i\frac\omega c (z_3 + \left<\bm\vartheta,\b y\right>)} \d \b z\d \b y \right]_j 
=  m_j (\omega,\bm\vartheta ),
\end{multline*}
where $m_j$ is given by \eqref{rightside}. Taking the Fourier transform of $\gb\psi$ with respect to space, we get
\begin{multline}\label{eqPert3}
\eta_j \left[ \bm\vartheta \times \left( \bm\vartheta \times \left( \tpsi (\tfrac\omega c (\bm\vartheta +\b e_3)) \b p \right)\right)\right.\\
+
\frac{\omega^2 }{c^2(2\pi)^3} \int_{\R^3}\int_{\Omega} \int_{\Omega} \bm\vartheta \times \left( \bm\vartheta \times \left( \bchi_0 \greentensor (\omega, \b y-\b z) \tpsi (\b k) \b p\right) \right)\e^{-\i\frac\omega c (z_3 + \left<\bm\vartheta,\b y\right>)}\e^{\i\left<\b k,\b z\right>} \d \b z\d \b y \d \b k \\
+ \left.\frac{\omega^2 }{c^2(2\pi)^3} \int_{\R^3}\int_{\Omega} \int_{\Omega} \bm\vartheta \times \left( \bm\vartheta \times \left( \tpsi (\b k) \greentensor (\omega, \b y-\b z) \bchi_0 \b p\right) \right)\e^{-\i\frac\omega c (z_3 + \left<\bm\vartheta,\b y\right>)} \e^{\i\left<\b k,\b y\right>} \d \b z\d \b y \d \b k \right]_j\\
= m_j (\omega,\bm\vartheta ) .
\end{multline}
This equation for $\bm{\tilde m} (\b v ) := \b m (\omega,\bm\vartheta ),$ using the definitions of the integral operators \eqref{Operators} admits the compact form \eqref{eqFinal}.
\end{proof}

Now, we are in position to formulate the inverse problem: Recover from the expressions
\[
\eta_j \left[ \bm\vartheta \times \left( \bm\vartheta \times \left( \left( \tpsi (\b v) + \bchi_0 \bm{\mathcal{K}} [\tpsi ] (\b v) + \bm{\mathcal{K}}^\dagger [\tpsi ] (\b v) \bchi_0 \right) \b p \right)\right)\right]_j , \quad j\in \{1,2\},
\]
the matrix-valued function $\bm\psi : \Omega \rightarrow \C^{3\times 3},$ if we assume that we have measurements for every incident polarization.

Let us now specify the polarization vectors $\bm\eta$ and $\b p.$ We choose two different incident polarization vectors $\b q^{(1)} = \b e_1$ and $\b q^{(2)} = \b e_2 ,$  and using the formulas \eqref{reference} we obtain the vectors
\begin{equation}\label{eqIncident2}
\begin{aligned}
\bm\eta^{(1)} &= \frac{\sqrt{2}}2 \begin{pmatrix} 1 \\1  \\0 \end{pmatrix} , & \b p^{(1)} &= \frac{1}2 \begin{pmatrix} 1 -\i\\ 1 +\i  \\ 0 \end{pmatrix}, \\
\bm\eta^{(2)} &= \frac{\sqrt{2}}2 \begin{pmatrix} \phantom{-}1 \\ -1  \\ \phantom{-}0 \end{pmatrix} , & \b p^{(2)} &= \frac{1}2 \begin{pmatrix} 1 +\i\\ 1 -\i \\ 0  \end{pmatrix}.
\end{aligned}
\end{equation}

\begin{remark}
To find, for instance, the form of the incident wave $\b p^{(1)} \hat f(\omega) \, \e^{-\i\tfrac{\omega}c x_3 },$ for $\omega >0,$ in the time domain we have to extend it for negative frequencies and consider its inverse Fourier transform. Then, we have
\begin{equation*}
\begin{aligned}
\b E^{(1)} (t,\b x) &=: \frac1{2\pi} \int_0^\infty \b p^{(1)} \hat f(\omega) \, \e^{-\i\tfrac{\omega}c x_3 } \e^{-\i \omega t} \d \omega +\frac1{2\pi} \int_{-\infty}^0 \overline{\b p^{(1)}} \hat f(\omega) \, \e^{-\i\tfrac{\omega}c x_3 } \e^{-\i \omega t} \d \omega \\
&= \frac1{2\pi} \int_0^\infty \b p^{(1)} \hat f(\omega) \, \e^{-\i\tfrac{\omega}c x_3 } \e^{-\i \omega t} \d \omega +\frac1{2\pi} \int_0^{\infty} \overline{\b p^{(1)} \hat f(\omega) \, \e^{-\i\tfrac{\omega}c x_3 } \e^{-\i \omega t}} \d \omega \\
&=\frac1{\pi} \Re \int_0^\infty \b p^{(1)} \hat f(\omega) \, \e^{-\i\tfrac{\omega}c x_3 } \e^{-\i \omega t} \d \omega
\end{aligned}
\end{equation*}
If the small spectrum is centered around a frequency $\nu,$ we approximate $\hat{f}(\omega) \simeq \delta (\omega - \nu),$ for $\omega >0,$ to obtain
\begin{equation*}
\begin{aligned}
\b E^{(1)} (t,\b x) &= \frac1{\pi} \Re  \left\{\b p^{(1)}  \e^{-\i \nu (\tfrac{x_3}{c}+t) } \right\} \\
&= \frac1{2\pi} \begin{pmatrix} \cos (\nu (\tfrac{x_3}{c}+t)) - \sin (\nu (\tfrac{x_3}{c}+t)) \\ \cos (\nu (\tfrac{x_3}{c}+t)) + \sin (\nu (\tfrac{x_3}{c}+t))  \\0 \end{pmatrix} \\
&= \frac1{\sqrt{2}\pi} \begin{pmatrix} \cos (\tfrac{\pi}4 +\nu (\tfrac{x_3}{c}+t))  \\ \sin (\tfrac{\pi}4 +\nu (\tfrac{x_3}{c}+t)) \\0 \end{pmatrix}.
\end{aligned} 
\end{equation*}
We see that $\b E^{(1)}$ describes also a circularly polarized wave with a phase shift. 

\end{remark}

If we neglect the zeroth third components, we observe that $\bm\eta^{(1)},\bm\eta^{(2)}  \in \R^2$ and $\b p^{(1)}, \b p^{(2)} \in \C^2$ form a basis in $\R^2$ and $\C^2,$ respectively. The following result shows that measurements for additional polarization vectors $\b q$ do not provide any further information.

\begin{proposition}\label{thReconstructableMatrixEntries}
Let $\bm\vartheta\in \mathbb{S}^2_+$ be fixed and $\b q = \b q^{(1)} , \, \b q^{(2)} .$ Then, the equation \eqref{eqFinal} is equivalent to the system of equations
\begin{equation}\label{eqMeasurementsPolarisationReduced}
\begin{aligned}
[\b P_{\bm\vartheta} \b Y \b p^{(1)}]_1 &= b^{(1)}_1 ,\qquad &
[\b P_{\bm\vartheta} \b Y \b p^{(1)}]_2 &= b^{(1)}_2 ,\\
[\b P_{\bm\vartheta} \b Y \b p^{(2)}]_1 &= b^{(2)}_1 ,\qquad & 
[\b P_{\bm\vartheta} \b Y \b p^{(2)}]_2 &= -b^{(2)}_2  ,
\end{aligned}
\end{equation}
where $\b Y := \tpsi (\b v) + \bchi_0 \bm{\mathcal{K}} [\tpsi ] (\b v) + \bm{\mathcal{K}}^\dagger [\tpsi ] (\b v) \bchi_0,$ $b^{(k)}_j := -\sqrt{2} \tilde m^{(k)}_j  , \, k,j=1,2,$ and $\b P_{\bm\vartheta}$ denotes the orthogonal projection in direction $\bm\vartheta$. The upper index on the data counts for the different incident polarisations.
\end{proposition}

\begin{proof}
The system of equations \eqref{eqFinal} for $(\b q,j)\in\{(\b q^{(1)} ,1),(\b q^{(1)} ,2),(\b q^{(2)},1),(\b q^{(2)},2)\}$ is equivalent to the four equations
\begin{equation}\label{eqFixed}
\begin{aligned}
\eta^{(1)}_1 [\bm\vartheta\times(\bm\vartheta\times \b Y \b p^{(1)})]_1 &=  \tilde m^{(1)}_1 ,\qquad &
\eta^{(1)}_2 [\bm\vartheta\times(\bm\vartheta\times \b Y \b p^{(1)})]_2 &=  \tilde m^{(1)}_2 ,\\
\eta^{(2)}_1 [\bm\vartheta\times(\bm\vartheta\times \b Y \b p^{(2)})]_1 &=  \tilde m^{(2)}_1 ,\qquad & 
\eta^{(2)}_2 [\bm\vartheta\times(\bm\vartheta\times \b Y \b p^{(2)})]_2 &= \tilde  m^{(2)}_2 .
\end{aligned}
\end{equation}
Indeed, for arbitrary polarisation $\b q = c_1 \b q^{(1)}+c_2 \b q^{(2)}$, $c_1 , c_2 \in \R$ the expression $\eta_j [\bm\vartheta\times(\bm\vartheta\times \b Y \b p)]_j$ can be written as a linear combination of the four expressions $\tilde m^{(k)}_j$, $k,j=1,2$:
\begin{align*}
\eta_1 [\bm\vartheta\times(\bm\vartheta\times \b Y \b p )]_1 &=  [c_1 \bm\eta^{(1)}+c_2 \bm\eta^{(2)}]_1 [\bm\vartheta\times(\bm\vartheta\times \b Y (c_1 \b p^{(1)}+c_2 \b p^{(2)}) )]_1\\
&=  c_1^2 \eta^{(1)}_1 [\bm\vartheta\times(\bm\vartheta\times \b Y \b p^{(1)} )]_1 + c_1 c_2 \eta^{(1)}_1 [\bm\vartheta\times(\bm\vartheta\times \b Y \b p^{(2)} )]_1 \\
&\phantom{=} + c_1 c_2 \eta^{(2)}_1 [\bm\vartheta\times(\bm\vartheta\times \b Y \b p^{(1)} )]_1 +c_2^2 \eta^{(2)}_1 [\bm\vartheta\times(\bm\vartheta\times \b Y \b p^{(2)} )]_1 \\
&=  c_1^2 \eta^{(1)}_1 [\bm\vartheta\times(\bm\vartheta\times \b Y \b p^{(1)} )]_1 + c_1 c_2 \eta^{(2)}_1 [\bm\vartheta\times(\bm\vartheta\times \b Y \b p^{(2)} )]_1 \\
&\phantom{=} + c_1 c_2 \eta^{(1)}_1 [\bm\vartheta\times(\bm\vartheta\times \b Y \b p^{(1)} )]_1 +c_2^2 \eta^{(2)}_1 [\bm\vartheta\times(\bm\vartheta\times \b Y \b p^{(2)} )]_1 \\
&= (c_1^2 + c_1 c_2) \tilde m^{(1)}_1 + (c_2^2 + c_1 c_2) \tilde m^{(2)}_1 ,
\end{align*}
and similarly
\[
\eta_2 [\bm\vartheta\times(\bm\vartheta\times \b Y \b p )]_2 = (c_1^2 - c_1 c_2) \tilde m^{(1)}_2 + (c_2^2 - c_1 c_2) \tilde m^{(2)}_2 .
\]
Decomposing $\b Y \b p=\left<\bm\vartheta,\b Y \b p\right>\bm\vartheta+\b P_{\bm\vartheta} \b Y\b p$, where $\b P_{\bm\vartheta}\in\R^{3\times3}$ denotes the orthogonal projection in direction $\vartheta$, and using that
\[ \bm\vartheta\times(\bm\vartheta\times \b Y \b p) = \bm\vartheta\times(\bm\vartheta\times \b P_{\bm\vartheta} \b Y \b p) = \left<\bm\vartheta,\b P_{\bm\vartheta} \b Y \b p\right>\bm\vartheta-\b P_{\bm\vartheta} \b Y \b p = -\b P_{\bm\vartheta} \b Y \b p, \] 
the system of equations \eqref{eqFixed} considering \eqref{eqIncident2} can be written in the form \eqref{eqMeasurementsPolarisationReduced}.
\end{proof}

For $\b Y (\b v)= \tpsi (\b v) + \bchi_0 \bm{\mathcal{K}} [\tpsi ] (\b v) + \bm{\mathcal{K}}^\dagger [\tpsi ] (\b v) \bchi_0$, where $\b v= \tfrac\omega c (\bm\vartheta +\b e_3),$ $\gb\vartheta\in \mathbb{S}^2_+$ ,  
\autoref{thReconstructableMatrixEntries} shows that the data $\tilde m^{(k)}_j (\b v)$ for $k,j=1,2$ and two different polarisation vectors $\b q = \b e_1$ and $\b q =\b e_2$ uniquely determine the projections
$ [\b P_{\bm\vartheta} \b Y \b p^{(k)}]_j$ for $k,j \in\{1,2\}.$ 

Moreover, measurements for additional polarisations $\b q$ do not provide any further informations so that at every detector point, corresponding to a direction $\bm\vartheta\in \mathbb{S}^2_+$, only the four elements $[\b P_{\bm\vartheta} \b Y \b p^{(k)}]_j$, $k,j=1,2$, of the projection influence the measurements.

\begin{remark}
In contrast to standard OCT where three polarisation vectors where needed \cite[Proposition 11]{ElbMinSch15} and to first order Born-approximation where $\b Y = \tpsi ,$ as we are going to see in the following, the above measurements due to the special form of $\b Y$ allow for reconstructing all the unknowns functions $\psi_{ij}. $
\end{remark}

\begin{proposition}
Let $\gb\vartheta\in \mathbb{S}^2_\ast :=\{\gb\mu\in \mathbb{S}^2\mid \mu_1 \neq \mu_2 , \, \mu_3>0\}.$ For two given incident polarisation vectors $\b q^{(1)}$ and $\b q^{(2)},$ the system of equations \eqref{eqMeasurementsPolarisationReduced} is equivalent to a Fredholm type system of integral equations
\begin{align}\label{eqProp1}
(\id + \bm{\mathcal{C}})
\begin{pmatrix}
\tilde{\psi}_{11} \\
\tilde{\psi}_{12} \\
\tilde{\psi}_{22}
\end{pmatrix} &= \b b,
\end{align}
for some compact operator $\bm{\mathcal{C}} : (L^2 (\Omega))^3 \rightarrow (L^2 (\mathbb{S}^2))^3 $ and known right hand side $\b b$ depending on the OCT data. Given the solution of \eqref{eqProp1}, the component $\tilde{\psi}_{33}$ satisfies a Fredholm integral equation of the first kind
\begin{equation}\label{eqProp2}
\mathcal{C} \tilde{\psi}_{33} = b,
\end{equation}
where $\mathcal{C} : L^2 (\Omega) \rightarrow L^2 (\mathbb{S}^2)$ is a compact operator and $b$ depends on the solution of \eqref{eqProp1}.
\end{proposition}

\begin{proof}
In order to reformulate equations \eqref{eqMeasurementsPolarisationReduced}, first we consider an arbitrary vector $\b p$ and we split the expression $\b P_{\bm\vartheta} \b Y \b p$ into the sum
\begin{equation}\label{eqDecomp}
\b P_{\bm\vartheta} \b Y \b p =  (\id - \bm\vartheta \bm\vartheta^T) \tpsi \b p + (\id - \bm\vartheta \bm\vartheta^T) \bchi_0 \bm{\mathcal{K}} [\tpsi] \b p + (\id - \bm\vartheta \bm\vartheta^T) \bm{\mathcal{K}}^\dagger [\tpsi] \, \bchi_0  \b p ,
\end{equation}
omitting for simplicity the $\b v$ dependence of the unknown $\tpsi.$

The first term on the right hand side admits the decomposition
\begin{equation*}
(\id - \bm\vartheta \bm\vartheta^T) \tpsi \b p = 
\begin{pmatrix} 
p_1 (1 - \vartheta_1^2) & - p_1 \vartheta_1 \vartheta_2 + p_2 (1 - \vartheta_1^2) & -p_2 \vartheta_1 \vartheta_2 \\
- p_1 \vartheta_1 \vartheta_2 & - p_2 \vartheta_1 \vartheta_2 + p_1 (1 - \vartheta_2^2) & p_2 (1 - \vartheta_2^2) \\
- p_1 \vartheta_1 \vartheta_3 & - p_1 \vartheta_2 \vartheta_3 - p_2 \vartheta_1 \vartheta_3 & - p_2 \vartheta_2 \vartheta_3 
\end{pmatrix} \begin{pmatrix}
\tilde{\psi}_{11} \\
\tilde{\psi}_{12} \\
\tilde{\psi}_{22}
\end{pmatrix} ,
\end{equation*}
where we observe the independence on $\tilde{\psi}_{33}.$ To analyse the other two terms, we consider \eqref{Operators} and define the operators acting now on the components of the matrix-valued function $\b f:$
\[
\mathcal{K}_{kj}  [f ] (\b v) : =  \int_{\R^3} \b [K^{\b z}]_{kj} (\b v; \b k)  f (\b k) \d \b k , \quad \mathcal{K}_{kj}^\dagger  [f ] (\b v) : =  \int_{\R^3} [K^{\b y}]_{kj} (\b v; \b k)  f (\b k)  \d \b k , \,\, k,j=1,2,3.
\]
Since we are interested only in the first two components of $\b P_{\bm\vartheta} \b Y \b p$ and the calculations are rather lengthy we are going to omit the third component in the following expressions. The second term on the right hand side of \eqref{eqDecomp} reads
\begin{equation*}
(\id - \bm\vartheta \bm\vartheta^T) \bchi_0 \bm{\mathcal{K}} [\tpsi] \b p = \chi_0
\begin{pmatrix} 
p_1 \mathcal{L}_{11}  & p_1 \mathcal{L}_{12} + p_2 \mathcal{L}_{11} & p_2 \mathcal{L}_{12} \\
p_1 \mathcal{L}_{21}  & p_1 \mathcal{L}_{22} + p_2 \mathcal{L}_{21} & p_2 \mathcal{L}_{22} \\
 \ast & \ast & \ast
\end{pmatrix} \begin{pmatrix}
\tilde{\psi}_{11} \\
\tilde{\psi}_{12} \\
\tilde{\psi}_{22}
\end{pmatrix} ,
\end{equation*}
where
\[
\mathcal{L}_{kj} := (1 -\vartheta_k^2 - \vartheta_1 \vartheta_2) (\mathcal K_{1j} +  \mathcal K_{2j}) - \vartheta_k \vartheta_3 \mathcal K_{3j}, \quad k,j=1,2.
\]
The only term where $\tilde{\psi}_{33}$ appears is the last one (as expected), namely
\begin{multline*}
(\id - \bm\vartheta \bm\vartheta^T) \bm{\mathcal{K}}^\dagger [\tpsi] \, \bchi_0  \b p  \\
{=}\chi_0 (p_1 + p_2)
\begin{pmatrix} 
(1- \vartheta_1^2) \mathcal{M}_1 & - \vartheta_1 \vartheta_2 \mathcal{M}_1 + (1- \vartheta_1^2) \mathcal{M}_2 & - \vartheta_1 \vartheta_2 \mathcal{M}_2 & - \vartheta_1 \vartheta_3 \mathcal{M}_3 \\
- \vartheta_1 \vartheta_2 \mathcal{M}_1 & (1- \vartheta_2^2) \mathcal{M}_1 - \vartheta_1 \vartheta_2 \mathcal{M}_2 &   (1- \vartheta_2^2) \mathcal{M}_2 & - \vartheta_2 \vartheta_3 \mathcal{M}_3 \\
 \ast & \ast & \ast & \ast
\end{pmatrix}\!\!\begin{pmatrix}
\tilde{\psi}_{11} \\
\tilde{\psi}_{12} \\
\tilde{\psi}_{22} \\
\tilde{\psi}_{33}
\end{pmatrix},
\end{multline*}
where 
\[
\mathcal{M}_j := \mathcal{K}_{j1}^\dagger + \mathcal{K}_{j2}^\dagger , \quad j=1,2,3.
\]

We can combine now all the above formulas to obtain
\begin{equation*}
\b P_{\bm\vartheta} \b Y \b p = \left( \bm{\mathcal{I}} (\b p) + \chi_0 \bm{\mathcal{L}} (\b p) + \chi_0 (p_1 + p_2 ) \bm{\mathcal{M}}  \right) \b y,
\end{equation*}
where
\begin{equation}\label{matrices}
\begin{aligned}
\bm{\mathcal{I}} (\b p) &=  \begin{pmatrix} 
p_1 (1 - \vartheta_1^2) & - p_1 \vartheta_1 \vartheta_2 + p_2 (1 - \vartheta_1^2) & -p_2 \vartheta_1 \vartheta_2  & 0\\
- p_1 \vartheta_1 \vartheta_2 & - p_2 \vartheta_1 \vartheta_2 + p_1 (1 - \vartheta_2^2) & p_2 (1 - \vartheta_2^2) & 0 \\
\ast & \ast & \ast & \ast 
\end{pmatrix} ,\\
\bm{\mathcal{L}} (\b p) &= \begin{pmatrix} 
p_1 \mathcal{L}_{11}  & p_1 \mathcal{L}_{12} + p_2 \mathcal{L}_{11} & p_2 \mathcal{L}_{12} & 0 \\
p_1 \mathcal{L}_{21}  & p_1 \mathcal{L}_{22} + p_2 \mathcal{L}_{21} & p_2 \mathcal{L}_{22} & 0\\
 \ast & \ast & \ast & \ast
\end{pmatrix} ,\\
\bm{\mathcal{M}} &= \begin{pmatrix} 
(1- \vartheta_1^2) \mathcal{M}_1 & - \vartheta_1 \vartheta_2 \mathcal{M}_1 + (1- \vartheta_1^2) \mathcal{M}_2 & - \vartheta_1 \vartheta_2 \mathcal{M}_2 & - \vartheta_1 \vartheta_3 \mathcal{M}_3 \\
- \vartheta_1 \vartheta_2 \mathcal{M}_1 & (1- \vartheta_2^2) \mathcal{M}_1 - \vartheta_1 \vartheta_2 \mathcal{M}_2 &   (1- \vartheta_2^2) \mathcal{M}_2 & - \vartheta_2 \vartheta_3 \mathcal{M}_3 \\
 \ast & \ast & \ast & \ast
\end{pmatrix},
\end{aligned}
\end{equation}
and
\[
\b y = \begin{pmatrix} \tilde{\psi}_{11} \,\,
\tilde{\psi}_{12} \,\,
\tilde{\psi}_{22} \,\,
\tilde{\psi}_{33}
\end{pmatrix}^T .
\]

Then, the system of equations \eqref{eqMeasurementsPolarisationReduced}, considering \eqref{eqIncident2} reads
\begin{subequations}\label{eqReduced}
\begin{align}
[( \bm{\mathcal{I}} (\b p^{(1)}) + \chi_0 \bm{\mathcal{L}} (\b p^{(1)}) + \chi_0  \bm{\mathcal{M}}  ) \b y]_1 &= b^{(1)}_1 , \label{eqReduced1}\\
[( \bm{\mathcal{I}} (\b p^{(1)}) + \chi_0 \bm{\mathcal{L}} (\b p^{(1)}) + \chi_0  \bm{\mathcal{M}}  ) \b y]_2 &= b^{(1)}_2 , \label{eqReduced2}\\
[( \bm{\mathcal{I}} (\b p^{(2)}) + \chi_0 \bm{\mathcal{L}} (\b p^{(2)}) + \chi_0  \bm{\mathcal{M}}  ) \b y]_1 &= b^{(2)}_1 , \label{eqReduced3}\\
[( \bm{\mathcal{I}} (\b p^{(2)}) + \chi_0 \bm{\mathcal{L}} (\b p^{(2)}) + \chi_0  \bm{\mathcal{M}}  ) \b y]_2 &= -b^{(2)}_2 .\label{eqReduced4}
\end{align}
\end{subequations}

We observe that in all equations the coefficient in front of the operator $\bm{\mathcal{M}}$ is the same, which is the only operator applying on the fourth component of $\b y.$ In addition, from \eqref{matrices}, we see that $\vartheta_2 \bm{\mathcal{M}}_{14} = \vartheta_1 \bm{\mathcal{M}}_{24}.$ Thus, in order to eliminate $y_4$ we reformulate the above system as follows: we subtract from equation \eqref{eqReduced1} the equation \eqref{eqReduced3}, from equation \eqref{eqReduced2} the equation \eqref{eqReduced4} and from $\vartheta_2 \cdot$\eqref{eqReduced1} the equation $\vartheta_1 \cdot$\eqref{eqReduced2}, resulting to
\begin{subequations}\label{eqReducedb}
\begin{align*}
[( \bm{\mathcal{I}} (\b p^{(1)} -\b p^{(2)} ) + \chi_0 \bm{\mathcal{L}} (\b p^{(1)}-\b p^{(2)})   ) \b y]_1 &= b^{(1)}_1 - b^{(2)}_1, \\
[( \bm{\mathcal{I}} (\b p^{(1)} -\b p^{(2)} )+ \chi_0 \bm{\mathcal{L}} (\b p^{(1)}-\b p^{(2)})   ) \b y]_2 &= b^{(1)}_2 + b^{(2)}_2, \\
\vartheta_2 [( \bm{\mathcal{I}} (\b p^{(1)}) + \chi_0 \bm{\mathcal{L}} (\b p^{(1)}) + \chi_0  \bm{\mathcal{M}}  ) \b y]_1 \\
- \vartheta_1 [( \bm{\mathcal{I}} (\b p^{(1)}) + \chi_0 \bm{\mathcal{L}} (\b p^{(1)}) + \chi_0  \bm{\mathcal{M}}  ) \b y]_2 &= \vartheta_2 b^{(1)}_1 - \vartheta_1 b^{(1)}_2 .
\end{align*}
\end{subequations}

The above system in compact form reads
\begin{equation}\label{eqLinear}
 (\bm{\tilde{\mathcal{I}}} + \bm{\mathcal{N}} ) \bm{\tilde{y}} = \bm{\tilde{b}},
\end{equation}
where
\begin{equation*}%\label{matrices2}
\begin{aligned}
\bm{\tilde{\mathcal{I}}} &=  \frac{\i}2 \begin{pmatrix} 
2(\vartheta_1^2 -1) & 2 (1+ \vartheta_1 \vartheta_2 -\vartheta_1^2 )  & -2\vartheta_1 \vartheta_2 \\
 2\vartheta_1 \vartheta_2 &  2(\vartheta_2^2 -\vartheta_1 \vartheta_2 - 1 )  &  2(1-\vartheta_2^2 )  \\
-\vartheta_2 (1+\i) & \vartheta_1 (\i +1) + \vartheta_2 (1-\i) & -\vartheta_1 (1-\i) 
\end{pmatrix} ,\\
\bm{\mathcal{N}} &= \i \chi_0  \begin{pmatrix} 
 -\mathcal{L}_{11}  &  \mathcal{L}_{11} - \mathcal{L}_{12} &  \mathcal{L}_{12} \\
 -\mathcal{L}_{21}  & \mathcal{L}_{21} - \mathcal{L}_{22} &  \mathcal{L}_{22} \\
 \phantom{-}\mathcal{N}_1 & \mathcal{N}_2 & \mathcal{N}_3 
\end{pmatrix}, \\
\bm{\tilde{y}}  &=  \begin{pmatrix}
y_1 \\ y_2 \\ y_3
\end{pmatrix}, \quad \bm{\tilde{b}} = \begin{pmatrix}
b^{(1)}_1 - b^{(2)}_1 \\ b^{(1)}_2 + b^{(2)}_2 \\ \vartheta_2 b^{(1)}_1 - \vartheta_1 b^{(1)}_2
\end{pmatrix} ,
\end{aligned}
\end{equation*}
and
\begin{align*}
\mathcal{N}_1 &:= \tfrac12  [(1+\i) 
(\vartheta_1 \vartheta_2\mathcal{L}_{21} -\vartheta_2^2 \mathcal{L}_{11} )-2\i \vartheta_2 \mathcal{M}_1 ]  ,\\
\mathcal{N}_2 &:= \tfrac12 [(1-\i)(\vartheta^2_2 \mathcal{L}_{11} - \vartheta_1 \vartheta_2 \mathcal{L}_{21}) - (1+\i) (\vartheta^2_2 \mathcal{L}_{12} - \vartheta_1 \vartheta_2 \mathcal{L}_{22}) -2\i \vartheta_2 \mathcal{M}_2 + 2\i \vartheta_1 \mathcal{M}_1 ] ,\\
\mathcal{N}_3 &:= \tfrac12  [(1-\i) 
 (\vartheta_2^2 \mathcal{L}_{12} -\vartheta_1 \vartheta_2 \mathcal{L}_{22}) +2\i \vartheta_1 \mathcal{M}_2 ] .
\end{align*}
We compute the determinant of $\bm{\tilde{\mathcal{I}}}$ which is given by
\begin{align*}
\det (\bm{\tilde{\mathcal{I}}}) &=  -\tfrac{\i}8 \left( -\vartheta_1^3 + \vartheta_1^2 \vartheta_2 - \vartheta_1 \vartheta_2^2 +\vartheta_1 + \vartheta_2^3 - \vartheta_2 \right) \\
&= -\tfrac{\i}8 (\vartheta_2 - \vartheta_1) (\vartheta_1^2 + \vartheta_2^2 -1).
\end{align*}
Recall that $\bm\vartheta\in \mathbb{S}^2_+ ,$ meaning $\vartheta_3 >0.$ Then, if in addition we impose that $\vartheta_1 \neq \vartheta_2 $ for all $\bm\vartheta\in \mathbb{S}^2_+ ,$ the matrix $\bm{\tilde{\mathcal{I}}}$ is invertible with 
$
\bm{\tilde{\mathcal{I}}}^{-1} = \det (\bm{\tilde{\mathcal{I}}})^{-1} \mbox{adj}  (\bm{\tilde{\mathcal{I}}}).
$
Then, equation \eqref{eqLinear} can be written in the form
\begin{equation}\label{eqLinearFinal}
 (\id + \bm{\tilde{\mathcal{I}}}^{-1}\bm{\mathcal{N}} ) \bm{\tilde{y}} = \bm{\tilde{\mathcal{I}}}^{-1}\bm{\tilde{b}},
\end{equation}
which is the Fredholm integral equation of the second kind \eqref{eqProp1}, for $\bm{\mathcal{C}} := \bm{\tilde{\mathcal{I}}}^{-1}\bm{\mathcal{N}} ,$ and $\b b := \bm{\tilde{\mathcal{I}}}^{-1}\bm{\tilde{b}}.$ 
Once \eqref{eqLinearFinal} is solved 
for $y_1 ,y_2$ and $y_3$ we can choose one of the four equations from the system \eqref{eqReduced} resulting to a Fredholm integral equation of the first kind for the unknown $y_4$  now:
\begin{equation*}
\mathcal{M}_3 y_4 = b,
\end{equation*}
for some known function $b,$ depending on $\bm{\tilde{y}}$ and $\bm{\tilde{b}}.$ This is equation \eqref{eqProp2} for $\mathcal{C} := \mathcal{M}_3 .$

To see the compactness of the integral operator $\bm{\mathcal{K}}$, we go back to the definition \eqref{Operators} and we consider the following decomposition:
\begin{equation*}
\begin{aligned}
\bm{\mathcal{K}} [\b f ] (\tfrac\omega c (\bm\vartheta +\b e_3))  &= \frac{\omega^2 }{c^2(2\pi)^3} \int_{\R^3}\int_{\Omega} \int_{\Omega}  \greentensor (\omega, \b y-\b z)\e^{-\i\frac\omega c (z_3 + \left<\bm\vartheta,\b y\right>)} \e^{\i\left<\b k,\b z\right>} \bm{\tilde{f}} (\b k) \d \b z\d \b y \d \b k  \\
&= \frac{\omega^2 }{c^2} \int_{\Omega} \int_{\Omega}  \greentensor (\omega, \b y-\b z)\e^{-\i\frac\omega c (z_3 + \left<\bm\vartheta,\b y\right>)}  \b f (\b z) \d \b z\d \b y  \\
&= \int_{\Omega} \e^{-\i\frac\omega c \left<\bm\vartheta,\b y\right>}  \left(\frac{\omega^2}{c^2} \id\int_{\Omega}\greenkernel (\omega, \b y-\b z)  \e^{-\i\frac\omega c z_3} \b f (\b z) \d \b z \right.\\
&\phantom{=}\left.+\grad\div\int_{\Omega}\greenkernel (\omega, \b y-\b z)  \e^{-\i\frac\omega c z_3} \b f (\b z) \d \b z \right)\d \b y  \\
&= \int_{\Omega} \e^{-\i\frac\omega c \left<\bm\vartheta,\b y\right>}  \left(\frac{\omega^2}{c^2} \id\int_{\Omega}\greenkernel (\omega, \b y-\b z)  \e^{-\i\frac\omega c z_3} \b f (\b z) \d \b z \right.\\
&\phantom{=}\left.+\grad\div\int_{\Omega}\greenkernel (0, \b y-\b z)  \e^{-\i\frac\omega c z_3} \b f (\b z) \d \b z 
\right.\\
&\phantom{=}\left.+\grad\div\int_{\Omega}\left(\greenkernel (\omega, \b y-\b z) - \greenkernel (0, \b y-\b z)\right) \e^{-\i\frac\omega c z_3} \b f (\b z) \d \b z  \right)\d \b y .
\end{aligned}
\end{equation*}
The above expression in compact form reads
\[
\bm{\mathcal{K}} [\b f ] (\b v) = \bm{\mathcal{F}} \left[ (\mathcal{G} +\bm{\mathcal{G}}_0+\bm{\mathcal{G}}_1) [\e^{-\i\frac\omega c z_3}\b f] \right] (\b v),
\]
for the operators
\begin{equation*}
\begin{aligned}
\bm{\mathcal{F}} [ f ] (\bm\theta) &:= \int_{\Omega} \e^{-\i\frac\omega c \left<\bm\vartheta,\b y\right>}   f (\b y) \d \b y ,\\
\mathcal{G} [ f ] (\b x) &:= \frac{\omega^2}{c^2} \int_{\Omega}\greenkernel (\omega, \b x-\b y)  f (\b y) \d \b y,\\
\bm{\mathcal{G}}_0 [\b f ] (\b x) &:= \grad\div \int_{\Omega}\greenkernel (0, \b x-\b y) \b f (\b y) \d \b y,\\
\bm{\mathcal{G}}_1 [\b f ] (\b x) &:= \grad\div \int_{\Omega}\left(\greenkernel (\omega, \b x-\b y) - \greenkernel (0, \b x-\b y)\right) \b f (\b y) \d \b y .
\end{aligned}
\end{equation*}
The operator $\bm{\mathcal{F}}: L^2 (\Omega) \rightarrow L^2 (\mathbb{S}^2)$ is a modification of the usual far-field operator with smooth kernel thus compact. The operators $\mathcal{G}: L^2 (\Omega) \rightarrow L^2 (\Omega)$ and $\bm{\mathcal{G}}_1 : (L^2 (\Omega))^{3\times 3} \rightarrow (L^2 (\Omega))^{3\times 3}$ are also compact due to their weakly singular kernels, see for instance \cite{ColKre98, Pot00}, and the operator $\bm{\mathcal{G}}_0 : (L^2 (\Omega))^{3\times 3} \rightarrow (L^2 (\Omega))^{3\times 3}$ is bounded \cite{ColPaiSyl07}. Thus $\bm{\mathcal{K}} 
: (L^2 (\Omega))^{3\times 3} \rightarrow (L^2 (\mathbb{S}^2))^{3\times 3}$
is also compact. The same arguments hold for $\bm{\mathcal{K}}^\dagger $ and then we can consider these properties also for the operators acting on the components of the matrix-valued function.
\end{proof}

\begin{remark}
Equation \eqref{eqProp2} reflects the ill-posedness of the inverse problem, due to the compactness of the integral operator. 
\end{remark}

%%%%%%%%%%%%%%%%%%%%%%%%%%%%%%
%%% Conclusions
%%%%%%%%%%%%%%%%%%%%%%%%%%%%%%
\section{Conclusions}
In this work we have formulated the inverse problem of recovering the electric susceptibility of a non-magnetic, inhomogeneous orthotropic medium, placed in a polarized-sensitive Optical Coherence Tomograph, as a system of Fredholm integral equations (both of first and second kind). Under the assumptions of a non-dispersive, weakly scattering medium with small background variations we have shown that we can reconstruct all the coefficients of the matrix-valued susceptibility, given the data for two different incident polarization vectors.

%%% End

%%%%%%%%%%%%%%%%%%%%%%%%%%%%%%
%%% Acknowledgements
%%%%%%%%%%%%%%%%%%%%%%%%%%%%%%
\section*{Acknowledgement}
The work of OS has been supported by the Austrian Science Fund (FWF), Project P26687-N25 (Interdisciplinary Coupled Physics Imaging).
%%% End

%%% End

%%%%%%%%%%%%%%%%%%%%%%%%%%%%%%
%%% References
%%%%%%%%%%%%%%%%%%%%%%%%%%%%%%
% \section*{References}
% \printbibliography[heading=none]
%  \bibliographystyle{plain} 

 % \bibliography{ElbMinSch17a_report.bib}

%%% End

\end{document}